\documentclass[11pt]{amsart}

\usepackage{amsmath}
\usepackage{amsfonts}
\usepackage{amssymb}
\usepackage{mathtools}

\date{}

\newcommand{\beqa}{\begin{eqnarray*}}
\newcommand{\eeqa}{\end{eqnarray*}}
\newcommand{\beqn}{\begin{eqnarray}}
\newcommand{\eeqn}{\end{eqnarray}}

\newcommand{\cG}{\mathbb{G}}

\newcommand{\e}{\varepsilon}

\newcommand{\de}{\delta}

\newcounter{cnt1}
\newcounter{cnt2}
\newcounter{cnt3}
\newcommand{\blr}{\begin{list}{$($\roman{cnt1}$)$}
 {\usecounter{cnt1} \setlength{\topsep}{0pt}
 \setlength{\itemsep}{0pt}}}
\newcommand{\bla}{\begin{list}{$($\alph{cnt2}$)$}
 {\usecounter{cnt2} \setlength{\topsep}{0pt}
 \setlength{\itemsep}{0pt}}}
\newcommand{\bln}{\begin{list}{$($\arabic{cnt3}$)$}
 {\usecounter{cnt3} \setlength{\topsep}{0pt}
 \setlength{\itemsep}{0pt}}}
\newcommand{\el}{\end{list}}

\newtheorem{thm}{Theorem}[section]
\newtheorem{prob}[thm]{Problem}

\newtheorem{cor}[thm]{Corollary}

\newtheorem{Def}[thm]{Definition}
\newtheorem{prop}[thm]{Proposition}
\newtheorem{rem}[thm]{Remark}
\newcommand{\Rem}{\begin{rem} \rm}
\newcommand{\bdfn}{\begin{Def} \rm}
\newcommand{\edfn}{\end{Def}}

\newcommand{\ba}{\begin{array}}
\newcommand{\ea}{\end{array}}

\newtheorem*{ack}{Acknowledgement}

\newcommand{\ben}{\begin{enumerate}}
\newcommand{\een}{\end{enumerate}}

\begin{document}
\sloppy

\title[Commutative \& noncommutative Gurariy spaces]{A note on commutative and\\ noncommutative Gurariy spaces}

\author[Bandyopadhyay]{Pradipta Bandyopadhyay}
\address[Pradipta Bandyopadhyay]{Stat--Math Division\\
Indian Statistical Institute\\ 203, B.~T. Road\\ Kolkata
700 108, India. E-mail~: {\tt pradipta@isical.ac.in}}

\author[Sensarma]{Aryaman Sensarma}
\address[Aryaman Sensarma]{Stat--Math Division\\
Indian Statistical Institute\\8th Mile, Mysore Road\\ Bangalore 560059\\India, E-mail~: {\tt aryamansensarma@gmail.com}}

\begin{abstract}
In this short note, we answer two questions about Gurariy spaces asked in the literature in the affirmative. We also prove the analogue of one of the results for the noncommutative Guariy space.
\end{abstract}

\subjclass[2010]{46B25; 46B20; 46L07; 46L25.}

\keywords{Gurariy space, Noncommutative Guariy space, Almost isometry, Almost isometric ideals, Completely bounded Banach-Mazur distance. \\ \hspace*{\fill} {\bf Version: \today}}

\dedicatory{Dedicated to the memory of Professor Sudipta Dutta!}

\maketitle

\section{Introduction}

In this short note, we show that if a Banach space $X$ is almost isometric to a Gurariy space, then $X$ is also a Gurariy space. This answers a question posed in \cite{R1} in the affirmative. We prove an analogue of this result for the noncommutative Gurariy space (Definition~\ref{def1}) discussed in \cite{O}, as well as for $\mathcal{E}$-Gurariy property (Definition~\ref{def2}), and $\mathcal{E}$-injective property (Definition~\ref{def3}) discussed in \cite{P}.

We also show that a Banach space $X$ is a Gurariy space if and only if it is an almost isometric ideal in every superspace in which it embeds as a hyperplane, thereby answering a question posed in \cite{R} in the affirmative.

We end this note with some more observations about Gurariy spaces.

\section{Almost isometry}

\bdfn
Let $X$ and $Y$ be Banach spaces. \bla
\item Let $\e > 0$. A linear operator $T : X \to Y$ is an $\e$-isometry if
 \[
 (1+\e)^{-1} \|x\| \leq \|T(x)\| \leq (1+\e) \|x\| \quad \mbox{for all } x \in X.
 \]
\item We say that $X$ and $Y$ are almost isometric if for every $\e > 0$, there is an onto $\e$-isometry $T : X \to Y$.
\item \cite{G} \label{G} A Banach space $X$ is called a \emph{Gurariy space} if for every $\e >0$, and every pair of finite dimensional spaces $E \subseteq F$, every isometry $T : E \to X$ extends to an $\e$-isometry $\widetilde{T} : F \to X$.
    \el
\edfn

In \cite{G}, Gurariy proved the existence of a separable Gurariy space. We refer to \cite{ASCGM,GK} and \cite[Chapter 4]{GL} for more information on Gurariy spaces and to \cite{BDS1,K} for some recent results. For example, the separable Gurariy space is unique upto isometry. Throughout this note, this space will be denoted by~$\cG$.

Here is our first result. This is related to \cite[Problem 9.2]{MP}. 

\begin{thm}\label{thm3}
If a Banach space $X$ is almost isometric to a Gurariy space $G$, then $X$ is also a Gurariy space.
\end{thm}

\begin{proof}
Let $\e >0$. Choose $\de>0$ such that $(1+\de)^3 \leq 1+\e$. Since $X$ is almost isometric to $G$, there exists an onto $\de$-isometry $T_1: X \to G$.

Let $E \subseteq F$ be finite dimensional spaces and let $T: E \to X$ be an isometry. Then $S = T_1 \circ T : E \to G$ is an $\de$-isometry.

Define a new norm on $E$ by $\|e\|_1 = \|S(e)\|$. Clearly, $\|\cdot\|_1$ is $\de$-equivalent to $\|\cdot\|$ on $E$. By \cite[Lemma 1.1] {GK}, we can extend this norm to $F$, still denoted by $\|\cdot\|_1$, so that it is $\de$-equivalent to $\|\cdot\|$ on $F$.

Now, $S : (E, \|\cdot\|_1) \to G$ is an isometry. Since $G$ is a Gurariy space, $S$ can be extended to a $\de$-isometry $\widetilde{S} : (F, \|\cdot\|_1) \to G$, that is,
\[
(1+\de )^{-1} \|z\|_1 \leq \|\widetilde{S}(z)\| \leq (1+\de ) \|z\|_1 \quad \mbox{for all } z \in F.
\]

Now coming back to $\|\cdot\|$ we see that
\[
(1+\de)^{-2} \|z\| \leq \|\widetilde{S}(z)\| \leq (1+\de)^2 \|z\| \quad \mbox{for all } z \in F.
\]

Now consider the map $\widetilde{T} = T_1^{-1} \circ \widetilde{S} : F \to X$. Clearly, $\widetilde{T}$ extends $T$ and
\[
(1+\e)^{-1} \|z\| \leq (1+\de)^{-3} \|z\| \leq \|\widetilde{T}(z)\| \leq (1+\de)^3 \|z\| \leq (1+\e )\|z\|
\]
for all $z \in F$. Hence $X$ is a Gurariy space.
\end{proof}

Since the separable Gurariy space is unique upto isometry, this answers the question posed in \cite[Remark 6]{R1} in the affirmative.

\begin{cor}\label{cor}
If a Banach space $X$ is almost isometric to $\cG$, then it is isometric to $\cG$.
\end{cor}

Combining \cite[Theorem 2.6]{BDS1} and Corollary~\ref{cor}, we get an improvement of \cite[Theorem 2.6]{BDS1}.

\begin{thm}
A (non-separable) Banach space $X$ is a Gurariy space if and only if every separable a.i.-ideal in $X$ is almost isometric to $\cG$.
\end{thm}

\section{The Noncommutative Analogue}

Noncommutative Gurariy space was introduced by Oikhberg \cite{O} within the framework of operator spaces. Recently Pisier \cite{P} has done an extensive study on the Gurariy property in noncommutative set up.

\bdfn \label{def1}
\bla
\item If $V$ and $W$ are operator spaces, we define the completely bounded (cb) Banach-Mazur distance $d_{{cb}}(V, W)$ by
\[
 d_{{cb}}(V, W) = \inf\{\|\phi\|\|\phi^{-1}\| : \phi : V \to W \mbox{ is a cb linear isomorphism}\}.
 \]

\item An operator space $X$ is called c-exact if, for every finite-dimensional subspace $E$ of $X$, and for every $\e>0$  there exists a subspace $F$ of $M_n(\mathbb{C})$ satisfying  $d_{{cb}}(E, F)< c+\e$.

\item \cite{O} A separable $1$-exact space $X$ is a \emph{non-commutative Gurariy space} if it satisfies the following condition:
    \begin{quote}
    suppose a finite dimensional operator space $F$ is $1$-exact, $E\subseteq F$, $E^{'} \subseteq X$ and $u: E \longrightarrow E^{'}$ is a cb-isomorphism. Then for any $\delta>0$ there exists a subspace $F^{'} \subseteq X$ containing $E^{'}$ and a linear map $\widetilde{u}: F \longrightarrow F^{'}$ such that $\widetilde{u}|_{E}=u$ and $\|\widetilde{u}\|_{{cb}} \|\widetilde{u}^{-1}\|_{{cb}} < (1+\delta) \|u\|_{{cb}} \|u^{-1}\|_{{cb}}$.
    \end{quote}
    \el
\edfn

In \cite{L}, the authors prove the main properties of the noncommutative Gurariy space. For example, the noncommutative Gurariy space is unique upto completely isometric isomorphism, and this space is universal for separable $1$-exact operator space. Throughout this note, this space will be denoted by $\mathbb{NG}$. $\mathbb{NG}$ can also be thought as the operator space analogue of the Cuntz algebra $\mathcal{O}_2$. $\mathbb{NG}$ is the first example of a separable $1$-exact operator space that is not an $M_n$-space for any $n\in \mathbb{N}$.

Now we prove the noncommutative analogue of Corollary~\ref{cor}. This is also a contribution to \cite[Problem 9.2]{MP}. We believe that this result also is of independent interest and will prove useful in future investigations.

\begin{thm}\label{NG}
If an operator space $X$ satisfies $d_{{cb}}(X,\mathbb{NG})=1$, then $X$ is completely isometrically isomorphic to $\mathbb{NG}$.
\end{thm}

\begin{proof}
 Let $\e>0$. Choose $\eta >0$ such that $(1+\eta)^2<1+\e$. Let $E \subseteq X$ be a finite-dimensional subspace. Since $d_{{cb}}(X,\mathbb{NG})=1$, there exists a finite dimensional subspace $E^{'} \subseteq \mathbb{NG}$ such that $d_{{cb}}(E, E^{'})\leq 1+\eta$. As $\mathbb{NG}$ is $1$-exact, there exists $F \subseteq M_n(\mathbb{C})$ such that $d_{{cb}}(E^{'},F)<1+\eta$. Therefore $d_{{cb}}(E,F)\leq d_{{cb}}(E, E^{'}) d_{{cb}}(E^{'},F)<(1+\eta)^2<1+\e$. So, $X$ is $1$-exact.

Choose $\delta>0$ such that $(1+\delta)^5<1+\e$. Since $d_{{cb}}(X,\mathbb{NG})=1$, there exists an onto $\delta$-cb isomorphism $v:X \longrightarrow \mathbb{NG}$, that is, $v_n: M_n(X) \longrightarrow M_n(\mathbb{NG})$ is a $\delta$-isomorphism for all $n \geq 1$ and $\|v\|=\sup_{n\geq 1}\|v_n\|< \infty$. Let $F$ be $1$-exact finite dimensional operator space. Consider $E \subseteq F$. Let $u:E \longrightarrow E^{'} \subseteq X$ is a cb isomorphism, that is,  $u_n: M_n(E) \longrightarrow M_n(E^{'})$ is an isomorphism for all $n\geq 1$ and $\|u\|=\sup_{n \geq 1}\|u_n\|< \infty$. Then $w=v \circ u$ is a cb-isomorphism onto $w(E)$. Since $\mathbb{NG}$ is non-commutative Gurariy space, there exists $\widetilde{w}: F \longrightarrow \mathbb{NG}$ linear map such that $\widetilde{w}|_{E}=w$, and $\|\widetilde{w}\|_{{cb}} \|\widetilde{w}^{-1}\|_{{cb}} < (1+\delta) \|w\|_{{cb}}\|w^{-1}\|_{{cb}}$. Define $\widetilde{u}=v^{-1} \circ \widetilde{w}$. Now $\widetilde{u}|_{E}=v^{-1} \circ \widetilde{w}|_{E}=v^{-1}\circ w=v^{-1}\circ v \circ u=u$. Here $\widetilde{u}$ is a cb-isomorphism from $F$ to $\widetilde{u}(F)$. Now,
\beqa
\|\widetilde{u}\|_{{cb}}\|\widetilde{u}^{-1}\|_{{cb}} & = & \|v^{-1} \circ \widetilde{w}\|_{{cb}}\| \widetilde{w}^{-1} \circ v\|_{{cb}}\leq \|v\|_{{cb}}\|v^{-1}\|_{{cb}}\|\widetilde{w}\|_{{cb}}\|\widetilde{w}^{-1}\|_{{cb}}\\
& \leq & (1+\delta)^{2} \|\widetilde{w}\|_{{cb}} \|\widetilde{w}^{-1}\|_{{cb}}.
\eeqa

Since $\mathbb{NG}$ is non-commutative Gurariy space,
\[
(1+\delta)^{2}\|\widetilde{w}\|_{{cb}}\|\widetilde{w}^{-1}\|_{{cb}} < (1+\delta)^{3} \|w\|_{{cb}}\|w^{-1}\|_{{cb}}.
\]
So,
\beqa
(1+\delta)^{3}\|w\|_{{cb}}\|w^{-1}\|_{{cb}} & = & (1+\delta)^{3}\|v \circ u\|_{{cb}}\|u^{-1}\circ v^{-1}\|_{{cb}}\\
& \leq & (1+\delta)^{3}\|v\|_{{cb}} \|v^{-1}\|_{{cb}} \|u\|_{{cb}}\|u^{-1}\|_{{cb}} \\
& \leq & (1+\delta)^{5}\|u\|_{{cb}} \|u^{-1}\|_{{cb}} < (1+\e)\|u\|_{{cb}}\|u^{-1}\|_{{cb}}.
\eeqa
Since non-commutative Gurariy space is unique upto complete isomtery \cite[Theorem 1.1]{L}, $X$ is completely isometrically isomorphic to $\mathbb{NG}$.

This completes the proof.
\end{proof}

\begin{Def} \rm \label{def2} \cite[Definition 4.4]{P}
Let $\mathcal{E}$ be a class of finite dimensional operator spaces. An operator space $X$ has the \emph{$\mathcal{E}$-Gurariy property} if for any $\e>0$ and for any pair of spaces $E \subseteq F$ with $F \in \mathcal{E}$ the following holds:
    \begin{quote}
for any injective linear map $u: E \longrightarrow X$ there exists an injective extension $\widetilde{u}:F\longrightarrow X$ such that $\|\widetilde{u}\|_{{cb}} \|\widetilde{u}^{-1}\|_{{cb}}\leq  (1+\e) \|u\|_{{cb}} \|u^{-1}\|_{{cb}}$.
    \end{quote}
\end{Def}

Using the same argument of Theorem~\ref{NG} we get the following
\begin{cor}
Let $X$ be an operator space has the $\mathcal{E}$-Gurariy property. If an operator space $Y$ satisfies $d_{{cb}}(Y,X)=1$, then $Y$ also has the $\mathcal{E}$-Gurariy property.
\end{cor}

\begin{Def} \rm \label{def3} \cite[Definition 4.5]{P}
Let $\mathcal{E}$ be a class of finite dimensional operator spaces. An operator space $X$ is called \emph{$\mathcal{E}$-injective} if for any $\e>0$, any $E \in \mathcal{E}$ and any $S \subseteq E$, any map $u:S \longrightarrow X$ admits an extension $\widetilde{u}: E \longrightarrow X$ with $\|\widetilde{u}\|_{{cb}}\leq (1+\e)\|u\|_{{cb}}$.
\end{Def}

\begin{prop}
Let $X$ be an $\mathcal{E}$-injective operator space. If an operator space $Y$ satisfies $d_{{cb}}(Y,X)=1$, then $Y$ is also $\mathcal{E}$-injective.
\end{prop}

\begin{proof}
Let $\e>0$. Choose $\delta >0$ such that $(1+\delta)^3<1+\e$.

Since $d_{{cb}}(Y,X)=1$, there exists an onto $\delta$-cb isomorphism $v:Y \longrightarrow X$, that is, $v_n: M_n(Y) \longrightarrow M_n(X)$ is a $\delta$-isomorphism for all $n \geq 1$ and $\|v\|=\sup_{n\geq 1}\|v_n\|< \infty$. Let $E \in \mathcal{E}$ and $S \subseteq E$. Let $u:S \longrightarrow Y$ be a cb linear map. Consider the cb linear map $w=v \circ u: S \longrightarrow X$. Since $X$ is $\mathcal{E}$-injective, there exists an extension $\widetilde{w}: E \longrightarrow X$ with $\|\widetilde{w}\|_{{cb}} \leq (1+\delta)\|w\|_{{cb}}$.

Define $\widetilde{u}=v^{-1}\circ \widetilde{w}: E \longrightarrow Y$ be a cb linear map. Now $\widetilde{u}|_{S}=v^{-1}\circ \widetilde{w}|_S=v^{-1}\circ v \circ u=u$. And
\beqa
\|\widetilde{u}\|_{{cb}}=\|v^{-1}\circ \widetilde{w}\|_{{cb}} \leq \|v^{-1}\|_{{cb}}\|\widetilde{w}\|_{{cb}}.
\eeqa
Since $X$ is $\mathcal{E}$-injective space,
\beqa
\|v^{-1}\|_{{cb}}\|\widetilde{w}\|_{{cb}} \leq (1+\delta)\|v^{-1}\|_{{cb}}\|w\|_{{cb}} = (1+\delta)\|v^{-1}\|_{{cb}}\|v \circ u\|_{{cb}}\\
 \leq  (1+\delta)\|v\|_{{cb}}\|v^{-1}\|_{{cb}}\|u\|_{{cb}}\leq (1+\delta)^3\|u\|_{{cb}}<(1+\e)\|u\|_{{cb}}.
\eeqa

Hence, $Y$ is also $\mathcal{E}$-injective.

This completes the proof.

\end{proof}

\section{Almost isometric ideals}

Now we return to Banach spaces again.

\bdfn
\label{ai}\cite{AVN}
A (closed linear) subspace $Y$ of a Banach space $X$ is an {\em almost isometric ideal} (a.i.-ideal) in $X$ if for every $\e >0$ and every finite dimensional subspace $E \subseteq X$ there exists $T : E \to Y$ such that
\newpage
\bla
\item[(i)] $Te = e$ for all $e \in Y \cap E$, and
\item[(ii)] $T$ is an $\e$-isometry.
\el
\edfn

It is known \cite[Theorem 4.3]{AVN} that if $X$ is a Gurariy space and $Y \subseteq X$ is an a.i.-ideal in $X$, then $Y$ is a Gurariy space. Indeed, $X$ is a Gurariy space if and only if it is an a.i.-ideal in every superspace containing it. Here we prove~:

\begin{prop} \label{prop2}
Let $X$ be a Gurariy space and $Y \subseteq X$ be an infinite dimensional subspace such that for every $x \in X \setminus Y$, $Y$ is an a.i.-ideal in ${\rm span}\{x, Y\}$. Then $Y$ is a Gurariy space.
\end{prop}

\begin{proof}
We first observe that to prove $Y$ is a Gurariy space, by \cite[Lemma 4.3.1]{GL}, it is enough to show that for every $\e >0$, and every pair of finite dimensional spaces $E \subseteq F$ with dim$(F/E) = 1$, every isometry $T : E \to Y$ extends to an $\e$-isometry $\widetilde{T} : F \to Y$.

Let $\e >0$. Choose $\delta >0$ such that $(1+\delta)^2 \leq 1+\e$.

Let $E \subseteq F$ be finite dimensional subspaces with dim$(F/E) = 1$, and $T: E \longrightarrow Y \subseteq X$ be a linear isometry. Since $X$ is a Guarariy space, there exists a linear extension $\widetilde{T}: F \longrightarrow X$ of $T$ with
\[
(1+\delta)^{-1}\|f\| \leq \|\widetilde{T}f\| \leq (1+\delta)\|f\| \mbox{ for all } f \in F.
\]

Since dim$(F/E) = 1$, there exists $f_0 \in F$ such that $F = {\rm span}\{f_0, E\}$. Let $x_0 = \widetilde{T} (f_0)$ and $H = \widetilde{T} (F)$. Then $H \subseteq {\rm span}\{x_0, Y\}$ is finite dimensional. Since $Y$ is an a.i.-ideal in ${\rm span}\{x_0, Y\}$, there exists $S : H \to Y$ such that
\bla
\item $Sh = h$ for all $h \in H \cap Y$, and
\item $(1+\delta)^{-1}\|h\| \leq \|Sh\| \leq (1+\delta)\|h\|$ for all $h \in H$.
\el
It follows that $S\circ \widetilde{T}: F \longrightarrow Y$ is a linear extension of $T$ satisfying
\[
(1+\e)^{-1}\|f\| \leq \|S \circ \widetilde{T}f\| \leq (1+\e)\|f\| \mbox{ for all } f \in F.
\]
Therefore $Y$ is a Gurariy space.
\end{proof}

This answers the question posed in \cite[Question 17]{R} in the affirmative. If we replace a.i.-ideals by ideals in the above, this characterises $L_1$-predual spaces and was observed in \cite{R2}.

\begin{cor}
A Banach space $X$ is a Gurariy space if and only if it is an a.i.-ideal in every superspace in which it embeds as a hyperplane.
\end{cor}

\begin{proof}
It is known that $X$ embeds isometrically into a Gurariy space $G$ of the same density \cite{GK}. By hypothesis, every $g \in G \setminus X$, $X$ is an a.i.-ideal in ${\rm span}\{g, X\}$. By Proposition~\ref{prop2}, $X$ is a Gurariy space.
\end{proof}

\section{Miscellaneous Results}

We end this note with few more properties of Gurariy spaces, which are of independent interest.

\subsection{$M$-summands}

The next observation is closely related to the work of~\cite{BDS1}.

\begin{rem} \rm
Let $M$ be an $M$-summand of the Gurariy space $\mathbb{G}$, that is, $\mathbb{G}=M\oplus_{\infty}N$, where $M$ and $N$ are both infinite dimensional. By \cite[Theorem 3]{R} we get $M\cong \mathbb{G}$ and $N \cong \mathbb{G}$. But we know that $\mathbb{G}\oplus_{\infty}\mathbb{G}$ is not a Gurariy space  (see e.g., \cite[Remark 9]{R}). Hence, there does not exist any $M$-summand in the Gurariy space $\mathbb{G}$.
\end{rem}

\begin{rem} \rm
Every separable $L_1$-predual space is isometric to a $1$-complemented subspace of $\mathbb{G}$ \cite{W}. Hence quotient of the Gurariy space $\cG$ need not be the Gurariy space $\cG$.
\end{rem}

\begin{rem} \rm
We know that $\mathbb{G} \subseteq \mathbb{G} \oplus_{\infty} \mathbb{G} \subseteq G$, where $\mathbb{G}$ is the separable Gurariy space and $G$ is a non-separable Gurariy space. Here $\mathbb{G}$ is an a.i-ideal in $G$ but $\mathbb{G} \oplus_{\infty} \mathbb{G}$ is not an a.i-ideal in $G$ \cite[Remark 9]{R} \cite[Theorem 4.3]{AVN}.
\end{rem}

\subsection{$M$-ideals}

\bdfn
A subspace $M$ of a Banach space $X$ is said to be an $M$-ideal in $X$ if there is a projection $P$ on $X^*$ with $\ker(P) = M^\perp$ and for all $x^* \in X^*$, $\|x^*\| = \|Px^*\| +
\|x^*-Px^*\|$.
\edfn

It is known \cite[Theorem 3]{R} that any $M$-ideal in $\cG$ is isometric to $\cG$, and hence, is an a.i.-ideal.

\begin{prop}
Let $M$ be a separable $M$-ideal of a non-separable Gurariy space $G$. Then $M$ is the Gurariy space $\cG$ and hence an a.i-ideal in $G$.
\end{prop}

\begin{proof}
Let $M$ be a separable $M$-ideal of a non-separable Gurariy space $G$. So there exists a separable a.i-ideal $Z$ such that $M \subseteq Z \subseteq G$ \cite{A}. Using \cite[Theorem 2.6]{BDS1} we get $Z$ is isometric to $\cG$. By \cite[Proposition 1.17]{HWW} $M$ is an $M$-ideal in $Z$. Now using \cite[Theorem 3]{R} we conclude that $M$ is the Gurariy space $\cG$, hence it is an a.i-ideal in $G$.

This completes the proof.
\end{proof}

\begin{prob}
Let $M$ be a non-separable $M$-ideal of a non-separable Gurariy space $G$. Is $M$ a Gurariy space?
\end{prob}

\begin{ack} \rm
The second author wishes to acknowledge the support received from NBHM postdoctoral fellowship, Department of Atomic Energy (DAE), Government of India (File No: 0204/3/2020/R$\&$D-II/2445).
\end{ack}

\end{document}